\documentclass[letterpaper,11pt]{article}

\usepackage[ top=3cm, bottom=3cm, left=2cm, right=2cm]{geometry}

\usepackage{amsmath, amssymb, latexsym, mathrsfs, color, enumerate, amsthm, slashed}

\usepackage{hyperref}

\usepackage{graphicx}
\usepackage{calc}
\usepackage[toc,page]{appendix}

\usepackage[dvipsnames]{xcolor}

\usepackage{tikz}
\usepackage{mathdots}
\usepackage{yhmath}
\usepackage{cancel}
\usepackage{color}
\usepackage{siunitx}
\usepackage{array}
\usepackage{multirow}

\usepackage{gensymb}
\usepackage{tabularx}
\usepackage{booktabs}
\usetikzlibrary{fadings}

\newlength{\depthofsumsign}
\newlength{\totalheightofsumsign}
\newlength{\heightanddepthofargument}

\makeatletter
\newcommand*{\DivideLengths}[2]{%
  \strip@pt\dimexpr\number\numexpr\number\dimexpr#1\relax*65536/\number\dimexpr#2\relax\relax sp\relax
}
\makeatother

\numberwithin{equation}{section}

\newcommand{\be}{\begin{equation}}
\newcommand{\ee}{\end{equation}}

\newcommand{\act}[1]{\overset{#1}{\rhd}}

\newtheorem{theorem}{Theorem}
\newtheorem{definition}{Definition}
\newtheorem{proposition}{Proposition}
\newtheorem{remark}{Remark}
\newtheorem{example}{Example}

\title{F-algebra--Rinehart Pairs and Super F-algebroids}

\author{John Alexander Cruz Morales, Javier A. Gutierrez and Alexander Torres-Gomez}

\date{\small \today}

\begin{document}
\maketitle
\begin{abstract}
We define F-algebra--Rinehart pairs and super F-algebroids and study the connection between them.
\end{abstract}

\tableofcontents 


\section{Introduction}
The concept of an F-algebroid structure over a vector bundle was introduced in  \cite{CMTG} in analogy to the idea of a Lie algebroid in order to study the notion of F-manifolds in more general vector bundles than tangent bundles.\\

It is known that the notion of Lie algebroid is related to that of a Lie-Rinehart pair (from now on denoted by  L-R pair, see the appendix A for its definition and some examples), in the sense that given a Lie algebroid is possible to construct a L-R pair by considering the algebra of sections of a vector bundle and the algebra of continuous functions over the manifold. Inspired by this situation in this note we define an F-algebra--Rinehart pair (from now on denoted by F-R pair) and showed that under some considerations an F-algebroid can be described algebraically as an F-R pair. This is the main result of the present note. \\

With the purpose of having a coherent definition of L-R pairs we develop further the theory of F-algebras and F-algebra modules. As far as we know, the notion of F-algebra module have not been studied in the literature and in this note we try to fill this gap. So, we start by defining the category of F-algebras. Then, we define the notion of a module over an F-algebra. We find these results interesting by their own and, moreover, they are also useful for a complete characterization of F-algebras (something which have not been fully developed). The notion of F-modules allows us to define the notion of F-R pair. This is an algebraic construction that generalizes the idea of L-R pair.  \\

Trying to related the F-algebroids of \cite{CMTG} with the F-R pairs introduced in this note we see the need of extending the concept of F-algebroid by considering supermanifolds instead of just manifolds. The idea of super F-algebroids that arises in this context is reached thanks to the consideration of the notion of a smooth F-algebroid as a ringed space. With this notion we can prove that a super F-algebroid can be considered as an F-R pair. \\

We end the note with some final remarks about some work in progress and future projects.

\section{F-Algebras}

An F-algebra is a generalization of a unital commutative Poisson algebra where the Leibniz property has been weakened. 

\begin{definition}
Let $\mathcal F$ be a vector space. An \textbf{F-algebra} is a triple $(\mathcal F, \circ, [\;, \;])$ where 
\begin{enumerate}[i)]
\item $(\mathcal F, \circ)$ is an associative commutative algebra with a unit $e$;

\item $(\mathcal F, [\; ,\;])$ is a Lie algebra;

\item for every $X,Y,W \in \mathcal F$, we define the ``Leibnizator" $L(X,Z, W)$ as 
\be \label{leib1}
L(X,Z, W):= [X, Z \circ W]-[X, Z] \circ W-Z \circ [X, W] \ ;
\ee

\item the Leibnizator is a derivation in its first entry, that is, 
\be \label{falg1}
L(X \circ Y, Z, W)= X \circ L(Y,Z, W)+Y \circ L(X,Z,W) \ .
\ee
\end{enumerate}
\end{definition}

\begin{remark} In \cite{HM} the Leibnizator is denoted by $P_X(Z,W)$. We think that they used the letter $P$ to denote it because they empathizes that (\ref{leib1}) measures the deviation of the structure $(\mathcal F, \circ, [\;, \;])$ from that of a Poisson algebra on $(\mathcal F, \circ)$. 
\end{remark}

\begin{proposition}\label{PO}
Every Poisson algebra is an F-algebra.
\end{proposition}
\begin{proof}
It is easy to see that if $L(X,Z,W)\equiv 0$, for every $X,Y,Z \in \mathcal F$, then the left hand side of (\ref{leib1}) is zero and we obtain the Leibniz property that makes the Lie bracket a Poisson bracket. Moreover, equation (\ref{falg1}) is trivially satisfied. 
\end{proof}

\begin{proposition}
Let $(A, \circ)$ be an associative commutative algebra with unit, together with an abelian Lie bracket $[\;, \;]_a$. Then the triple $(A, \circ, [\;, \;]_a)$ is an F-algebra.
\end{proposition}
\begin{proof}
Note that the Leibnizator defined in terms of $[\;, \;]_a$ vanishes because every term on it is zero. Thus, equation (\ref{falg1}) is trivially satisfied (you get zero on both the left hand side and the right hand side).
\end{proof}

A homomorphism between F-algebras is defined in a natural way.

\begin{definition}
Let $(\mathcal F_1, \circ_1, [\;, \;]_1)$ and $(\mathcal F_2, \circ_2, [\;, \;]_2)$ be two F-algebras. A homomorphism $\rho: (\mathcal F_1, \circ_1, [\;, \;]_1) \to (\mathcal F_2, \circ_2, [\;, \;]_2)$ is a map between $\mathcal F_1$ and $\mathcal F_2$ which is at the same time a homomorphism of the associative commutative algebras $(\mathcal F_1, \circ_1) \to (\mathcal F_2, \circ_2)$ and a Lie algebras  homomorphism $(\mathcal F_1, [\;, \;]_1) \to (\mathcal F_2, [\;, \;]_2)$.

\end{definition}

Then, we can define the category of F-algebras as the category whose objects are F-algebras and its morphism F-algebras homomorphism. We will denote this category by $\mathbf{Falg}$. \\

Motivated by the theory of Frobenius manifolds and the theory of F-manifolds, where F-algebras play a central role, we want to study direct sum and tensor product of F-algebras. Although the direct sum and the tensor product of F-algebras, with the natural definitions of the product and the bracket given bellow, are not in general F-algebras they satisfy almost all the properties (missing only one in each case).\\

Let $(\mathcal F_1, \circ_1, [\;, \;]_1)$ and $(\mathcal F_2, \circ_2, [\;, \;]_2)$ be two F-algebras and  let $X_1, Y_1, Z_1 \in \mathcal F_1$ and $X_2, Y_2, Z_2 \in \mathcal F_2$.

\paragraph{Direct Sum of F-algebras.}  Let us define the associative commutative product $\circ$ and the bracket $[\;, \;]_\oplus$ on the direct sum $\mathcal F_1 \oplus \mathcal F_2$ as
\be
(X_1\oplus X_2) \circ (Y_1\oplus Y_2):=(X_1 \circ_1 Y_1) \oplus (X_2 \circ_2 Y_2) \ ,
\ee
and
\be
[X_1\oplus X_2, Y_1\oplus Y_2]_\oplus:=[X_1,Y_1]_1 \oplus [X_2, Y_2]_2 \ .
\ee
Moreover, the Leibnizator $L$ on $\mathcal F_1 \oplus \mathcal F_2$ can be written as a direct sum of the Leibnizators $L_1$ on $\mathcal F_1$ and $L_2$ on $\mathcal F_2$, that is, 
\be
L(X_1 \oplus X_2, Y_1 \oplus Y_2, Z_1 \oplus Z_2)=L(X_1, Y_1, Z_1) \oplus L(X_2,Y_2, Z_2) \ .
\ee
Using these definitions is easy to check that: the product $\circ$ is commutative and associative; the bracket $[\; ,\; ]_\oplus$ satisfies Jacobi identity, and it is symmetric; the Leibnizator $L$ is a derivation in the first entry.\\

Thus, the triplet $(\mathcal F_1 \oplus \mathcal F_2, \circ, [\;, \;]_\oplus)$ is not an F-algebra because the bracket $[\; , \; ]$ is not anti-symmetric (but symmetric instead).

\paragraph{Tensor Product of F-algebras.} In this case, let us define the associative commutative product $\circ$ and the Lie bracket $[\;, \;]_\otimes$ on the tensor product $\mathcal F_1 \otimes \mathcal F_2$ as
\be
(X_1\otimes X_2) \circ (Y_1\otimes Y_2):=(X_1 \circ_1 Y_1) \otimes (X_2 \circ_2 Y_2) \ ,
\ee
and
\be
[X_1\otimes X_2, Y_1\otimes Y_2]_\otimes:=[X_1,Y_1]_1 \otimes (X _2 \circ_2 Y_2) + (X_1 \circ Y_1)\otimes [X_2, Y_2]_2 \ .
\ee
Moreover, the Leibnizator $L$ on $\mathcal F_1 \otimes \mathcal F_2$ can be written in terms of the Leibnizators $L_1$ on $\mathcal F_1$ and $L_2$ on $\mathcal F_2$ as
\be\label{tenlei}
L(X_1\otimes X_2, Y_1\otimes Y_2, Z_1\otimes Z_2)= L_1(X_1,Y_1,Z_1)\otimes (X_2 \circ_2 Y_2 \circ_2 Z_2)+(X_1 \circ_1 Y_1 \circ_1 Z_1) \otimes L_2(X_2,Y_2,Z_2) \ .
\ee 

From these definitions it follows that: the product $\circ$ is commutative and associative; the bracket $[\; , \; ]_\otimes$ is anti-symmetric; the Leibnizator $L$ is a derivation in the first entry. The bracket $[\; , \; ]_\otimes$ does not satisfy, in general, Jacobi identity.\\

Let us denote the composition of terms that have to vanish on the Jacobi identity by $\text{Jacobi}_\otimes(X_1 \otimes X_2, Y_1 \otimes Y_2, Z_1 \otimes Z_2)$, that is, 
\begin{align}
\text{Jacobi}_\otimes &(X_1 \otimes X_2, Y_1 \otimes Y_2, Z_1 \otimes Z_2)= \notag\\
&[X_1 \otimes X_2, [Y_1 \otimes Y_2, Z_1 \otimes Z_2]_\otimes ]_\otimes+ [Y_1 \otimes Y_2, [Z_1 \otimes Z_2, X_1 \otimes X_2]_\otimes ]_\otimes + [Z_1 \otimes Z_2, [X_1 \otimes X_2, Y_1 \otimes Y_2]_\otimes ]_\otimes \ .
\end{align}
Thus, we obtain the following equation, instead of the Jacobi identity,
\begin{align} \label{tenjac}
&\text{Jacobi}_\otimes(X_1 \otimes X_2, Y_1 \otimes Y_2, Z_1 \otimes Z_2)= \notag \\
&L_1(X_1,Y_1, Z_1) \otimes \left(  X_2 \circ_2 [Y_2, Z_2]_2  \right)+ L_1(Y_1,Z_1, X_1) \otimes \left(  Y_2 \circ_2 [Z_2, X_2]_2  \right)+ L_1(Z_1,X_1, Y_1) \otimes \left(  Z_2 \circ_2 [X_2, Y_2]_2  \right)+ \notag \\
&  \left(  X_1 \circ_1 [Y_1, Z_1]_1  \right)  \otimes L_2(X_2,Y_2, Z_2)+ \left(  Y_1 \circ_1 [Z_1, X_1]_1  \right)  \otimes L_2(Y_2,Z_2, X_2)+\left( Z_1 \circ_1 [X_1, Y_1]_1 \right)  \otimes L_2(Z_2,X_2, Y_2)  \ .
\end{align}

Then, the triplet $(\mathcal F_1 \otimes \mathcal F_2, \circ, [\;, \;]_\otimes)$ is not, in general, an F-algebra because the bracket $[\; , \; ]_\otimes$ does not satisfy Jacobi identity. However, there are some cases where the tensor product of two F-algebras is again an F-algebra.\\

\begin{proposition}\label{tenpoi}
Let $\mathcal{F}_1$ and $\mathcal{F}_2$ be two F-algebras with trivial Leibnizator, i.e. they are Poisson algebras. Then $\mathcal{F}_1 \otimes \mathcal{F}_2$ is an F-algebra. In fact, $\mathcal{F}_1 \otimes \mathcal{F}_2$ is a Poisson algebra.
\end{proposition}

\begin{proof}
Note that in the case of Poisson algebras we have from equation (\ref{tenjac}) that $\text{Jacobi}_\otimes(X_1 \otimes X_2, Y_1 \otimes Y_2, Z_1 \otimes Z_2)=0$, because $L_1\equiv 0$ and $L_2 \equiv 0$. Moreover, it follows from the definition of the Leibnizator for the tensor product that, in this case, the Leibnizator is zero. Thus, the tensor product is a Poisson algebra. 
\end{proof}

\begin{remark} This proposition recovers, in the language of F-algebras, the known result stating that the tensor product of Poisson algebras is a Poisson algebra. In the context of category theory this means that the category of Poisson algebras is a monoidal category.
\end{remark}

It is natural to ask whether the condition of being Poisson is a necessary condition for the tensor product of two F-algebras to be an F-algebra. The following proposition gives a negative answer for this question. 

\begin{proposition} \label{tenabe}
Let $\mathcal{F}_1$ be an arbitrary F-algebra and $\mathcal{F}_2$ an F-algebra such that its Lie bracket vanishes (that is, it is an abelian Lie bracket). Then $\mathcal{F}_1 \otimes \mathcal{F}_2$ is an F-algebra. 
\end{proposition}

\begin{proof}
Since the Lie bracket for $\mathcal{F}_2$ vanishes, and then $L_2=0$, it follows from equation (\ref{tenjac}) that $\text{Jacobi}_\otimes(X_1 \otimes X_2, Y_1 \otimes Y_2, Z_1 \otimes Z_2)=0$.
\end{proof}

Thanks to the propositions above, we would like to formulate the following problem: \\

{\bf Problem 0}. What are the conditions on the F-algebras $\mathcal{F}_1$ and $\mathcal{F}_2$ such that their tensor product $\mathcal{F}_1 \otimes \mathcal{F}_2$ is an F-algebra?\\

Because of the form of equations (\ref{tenlei}) and (\ref{tenjac}), for the Leibnizator and the ``weakly Jacobi identity" on the tensor product $\mathcal F_1 \otimes \mathcal F_2$, and inspired on Proposition \ref{tenpoi} and Proposition \ref{tenabe} we conjecture the following answer to the problem above\\

{\bf Conjecture}. The tensor product $\mathcal F_1 \otimes \mathcal F_2$ of two F-algebras is an F-algebra only when the F-algebras are of the type described in Proposition \ref{tenpoi} or Propositon \ref{tenabe}.\\

Now, the category $\mathbf{Falg}$ has a ``geometric'' dual. Before giving a description of it let us first define what we mean by an F-algebra ideal. An F-alg ideal of an F-algebra $\mathcal F$ is an ideal with respect to the product and the Lie bracket defined on $\mathcal F$. We will refer to it just as an ideal. Note that the spectrum $\mathbf{Spec \, \mathcal F}$ is the set of its maximal ideals. Moreover, using standard arguments in Algebraic Geometry we can give $\mathbf{Spec \, \mathcal F}$ some topology. \\

Define $Pois \,  \mathcal F = \{ X \in \mathcal F: \text{the Leibnizator vanishes, that is, } L \equiv 0\}$. It was shown in \cite{HM} that $Pois \, \mathcal F$ is a Poisson subalgebra of $\mathcal F$. Consider $\mathbf{Spec\, Pois \, \mathcal F }$ with its topology. Then, there exists a surjective map 

\begin{center}
$\Upsilon : \mathbf{Spec \, \mathcal F} \rightarrow \mathbf{Spec \, Pois \, \mathcal F}\rightarrow 0 \, .$
\end{center}

\begin{proposition}\label{fib}
Fibers of the map $\Upsilon$ are spectra of Poisson algebras.
\end{proposition}

\begin{proof}
Let $\mathcal{I} \in \mathbf{Spec\, Pois \, \mathcal F}$ and $\Upsilon^{-1}(\mathcal{I})$. For each $\mathfrak{I} \in \Upsilon^{-1}(\mathcal{I})$ consider the quotient $\mathcal F_{\mathfrak{I}} = \mathfrak{I}/\mathcal{I}$. This is an F-algebra with bracket [$a + \mathcal{I}, b + \mathcal{I}$] = [$a,b$] + $\mathcal{I}$. Since $\mathcal{I} \subset Pois \, \mathcal F$, then the Leibnizator is trivial, so we have a Poisson algebra. 
\end{proof}

\subsection{Super F-algebras}

The structure of F-algebra can be easily extended over a super vector space (or in other words, over a $\mathbb Z_2$ graded vector space). Denoting by $| X |$ the parity of a vector $X$ we have:

\begin{definition}
Let $\mathcal F= \mathcal F_0 \oplus \mathcal F_1$ be a super vector space, and $X,Y,W, Z  \in \mathcal F$. A super \textbf{F-algebra} is a triplet $(\mathcal F, \circ, [\;, \;])$ where 
\begin{enumerate}[i)]
\item $(\mathcal F, \circ)$ is an associative commutative super algebra with a unit $e$, that is
\begin{align*}
X \circ Y= &(-1)^{|X| |Y|} Y \circ X , , & (X \circ Y) \circ Z&=X \circ (Y \circ Z) \, , & X \circ e=& e \circ X \ ;
\end{align*}

\item $(\mathcal F, [\; ,\;])$ is a Lie super algebra, that is we have $[X,Y]=- (-1)^{|X| |Y|} [Y, X]$ and 

\begin{align*}
 (-1)^{|X| |Z|} [X, [Y, Z]]+(-1)^{|Y| |Z|} [Y, [Z, X]]+(-1)^{|Z| |Y|} [Z, [X, Y]]=0 \, ;
\end{align*}

\item defining the ``Leibnizator" $L(X,Z, W)$ as 
\be \label{leib}
L(X,Z, W):= [X, Z \circ W]-[X, Z] \circ W-  (-1)^{|X| |Z|} Z \circ [X, W] \ ;
\ee

\item the Leibnizator is a super derivation in its first entry, that is, 
\be \label{falg}
L(X \circ Y, Z, W)= X \circ L(Y,Z, W)+  (-1)^{|X| |Y|} Y \circ L(X,Z,W) \ .
\ee
\end{enumerate}
\end{definition}

Similarly we can define homomorphisms of super F-algebras and the category of super F-algebras. We also have a ``geometric'' dual of this category and an analogue of Proposition \ref{fib} holds in this case. 

\subsection{F-algebra Modules}
\begin{definition}
A vector space V is said to be a (left) \textbf{F-algebra module} if
\begin{enumerate}[i)]

\item V is a Lie algebra module, that is, for every $X, Y \in \mathcal F$ and $v \in V$ we have
\begin{align*}
\mathcal F \times V &\to V  \\
(X, v)  & \mapsto X \act{l} v
\end{align*}
satisfying 
\[  [X,Y] \act{l} v =X \act{l}  (Y \act{l } v)  - Y \act{l}  (X \act{l } v)  \ ; \]

\item V is a unital associative commutative algebra module, that is, for every $X, Y \in \mathcal F$ and $v \in V$ we have
\begin{align*}
\mathcal F  \times V &\to V  \\
(X, v)  & \mapsto X \act{a} v
\end{align*}
satisfying
\begin{align*}
(X \circ Y) \act{a} v =&\frac 12 \left(X \act{a}(Y \act{a} v) + Y\act{a}(X \act{a} v) \right) \ , &  e \act{a} v=& v \ ;
\end{align*}

\item defining $\widetilde L(X,Y,v) \in V$ as
\be \label{def.LXYv}
\widetilde L(X,Y,v):= X \act{l} (Y \act{a} v)- [X,Y] \act{a} v - Y \act{a} ( X \act{l} v ) \ , 
\ee
it has to satisfy, for every $X, Y, Z \in \mathcal F$ and $v \in V$,
\be \label{con.LXYv}
\widetilde L(X \circ Y, Z,v)= X \act{a} \widetilde L(Y,Z,v)+ Y \act{a} \widetilde L(X,Z,v)    \ ;
\ee

\item defining $\overline L(v, X,Y) \in V$ as
\be \label{def.LvXY}
\overline L(v,X,Y):=  ( X  \circ Y) \act{l} v - X \act{a} (Y \act{l} v)  - Y \act{a} ( X \act{l} v)  \ , 
\ee
it has to satisfy, for every $X, Y, Z \in \mathcal F$ and $v \in V$, 
\be \label{con.LvXY}
  \overline L(Z \act{a} v, X,Y) =  L(Z,X,Y) \act{a} v+ Z \act{a} \overline L(v,X,Y)    \ .
 \ee

\end{enumerate}

\end{definition}

\begin{proposition}
A Poisson module is an F-algebra module. 
\end{proposition}
\begin{proof}
We can see that if  $\widetilde L(X,Y,v) \equiv 0$, $\overline L(v,X,Y) \equiv 0$ and $L(X,Y,Z) \equiv 0$, then the left hand side of (\ref{def.LXYv}) and (\ref{def.LvXY}) is zero, recovering the conditions that a Poisson algebra module has to satisfy, see \cite{O, J}. Moreover, equation (\ref{con.LXYv}) and (\ref{con.LvXY}) are trivially satisfied.
\end{proof}

Another interesting example is given by the so-called Adjoint Module. If we take the vector space $V$ as the F-algebra itself with the Lie algebra action $X \act{l}$ as $[X, \;]$ and the associative commutative algebra action $X \act{a}$ as $X \circ$ then $\widetilde L$ and $\overline{L}$ becomes the Leibnizator and they satisfy the property of an F-algebra that it is a derivation in the first entry. \\

We can also define the category of F-algebra modules as the category of $k$-linear representations of an F-algebra, where $k$ is a field of characteristic 0. The notion of F-algebra module is crucial for defining F-algebra-Rinehart pairs. This will be done in the next section.

\section{F-algebra--Rinehart Pair}

The definition of an F-R pair is an extension of the definition of a L-R pair (see the appendix \ref{LRpair} for the definition of L-R pairs).

\subsection{Definition}

\begin{definition}
An F-R pair denoted by  $(\mathcal F, C)$ consist of 
\begin{itemize}
 
\item an F-algebra $\mathcal F$ with Lie bracket $[\; , \;]$ and associative commutative product $\circ$, 

\item a commutative associative algebra $C$ with product $\cdot$, 
 
\end{itemize}
such that 
\begin{enumerate}[i)]
\item $C$ is an $\mathcal F$-module, with the actions of an element $X \in \mathcal F$ on an element $f \in C$ denoted by $X \act{l} f$ and $X \act{a} f$,

\item $\mathcal F$ is a $C$-module, with the action of $f \in C$ on $X \in \mathcal F$ denoted by $f \act{c} X$,  and defining $$\hat{L}(X,f,Y):=[X,f\act{c}Y]-(X\act{l} f)\act{c}Y+f\act{c}[X,Y] \ ,$$ 

it has to satisfy
 
 \[\hat{L}(X,f\cdot g,Y)=f\act{c}\hat{L}(X, g,Y) + g\act{c}\hat{L}(X, f,Y) \ .\]

\end{enumerate}
These data have the following compatibility between modules
\begin{enumerate}
\item compatibility between the $\mathcal F$-algebra products and the commutative associative algebra action

\begin{enumerate}

\item compatibility between $[\;, \;  ]$ and $\act{c}$ 

\[ [X, f \act{c} Y ]= (X \act{l} f) \act{c} Y + f \act{c}[X, Y] \ , \]

\item compatibility between $\circ$ and $\act{c}$,

\[  X \circ ( f \act{c} Y)=  \frac 12 \left( (X \act{a} f) \act{c} Y +f \act{c} (X \circ Y) \right) \]

\end{enumerate}

\item compatibility between the associative product $\cdot$ and $F$-algebra actions

\begin{enumerate}

\item compatibility between $\cdot$ and $\act{l}$ 

\[ f \cdot (X \act{l} g)= (f \act{c}X) \act{l} g \ ,\]

\item compatibility between $\cdot$ and $\act{a}$ 

\[ f \cdot (X \act{a} g)= (f \act{c}X) \act{a} g  \ ,\]

\end{enumerate}

\end{enumerate}

\end{definition}

\subsection{Examples}
Here we give some examples of  F-R pairs.
\begin{example}
For any F-algebra $\mathcal F$, let the commutative algebra $C$ be equal to that F-algebra, that is $C=\mathcal F$. The actions are the corresponding adjoint actions of $\mathcal F$ over $\mathcal F$.\\
\end{example}

\begin{example}
An F-manifold is a manifold $M$ with the structure of an F-algebra on sections $\Gamma(TM)$ of the tangent bundle $TM$. Now, consider the pair $\mathcal F=\Gamma(TM)$ and $\mathcal C=C^\infty(M)$ with $\act{l}$ the usual actions of vector fields on smooth function as derivations and $\act{c}$ the usual multiplication of function and vector fields. Then, if the action $\act{a}$ is the trivial one (mapping any function to itself)  it is easy to check that we obtain an F-R pair.\\
\end{example}





\section{Super F-algebroid}

An F-algebroid was defined in \cite{CMTG}, as a generalization of the concept of an F-manifold (there were also given some examples and applications there). We recall the definition from \cite{CMTG}:\\

\begin{definition}
Let $TM$ be the tangent bundle of a smooth manifold $M$;  $[ \; , \; ]$ a Lie bracket between section of $TM$;  $\circ$ denote a multiplicative commutative and associative structure on section of $TM$; and  $e \in \Gamma(TM)$ a unit vector field with respect to the multiplicative structure $\circ$. An F-algebroid is a  5-tuple $(E, [ \; , \; ]_E, \diamond, \mathcal U, \rho)$ where
\begin{enumerate}
\item
 $E$ is a vector bundle over a manifold $M$;
 \item
 $[\; , \;  ]_E: \Gamma(E) \times \Gamma(E) \to \Gamma(E)$ is a Lie bracket on sections of the vector bundle, that is, it is anti-symmetric and satisfies Jacobi identity $[\alpha, [\beta, \gamma]_E]_E+[\beta,[\gamma, \alpha ]_E]_E+ [\gamma, [\alpha,\beta ]_E]_E=0$;

\item $\diamond: \Gamma(E) \times \Gamma(E) \to \Gamma(E)$ is a commutative and associative multiplicative structure on sections of the vector bundle, that is, $\alpha \diamond \beta= \beta \diamond \alpha$ and $(\alpha \diamond \beta) \diamond \gamma= \alpha \diamond (\beta \diamond \gamma)$;

\item $\mathcal U \in \Gamma(E)$ is a unit with respect to the multiplicative structure $\diamond$, that is,  $\mathcal U \diamond \alpha= \alpha=\alpha \diamond  \mathcal U$;

\item
  $\rho: \Gamma(E) \to \Gamma(TM)$, called the anchor map, relates sections of the vector bundle to sections of the tangent bundle;
\end{enumerate}
satisfying the following conditions for $\alpha, \beta, \gamma, \tau \in \Gamma(E)$  and $f \in C^\infty(M)$:
\begin{enumerate}[a)]
\item defining the Leibizator $\mathcal L (\alpha, \beta, \gamma)$ as
\[
\mathcal L (\alpha, \beta, \gamma):= [\alpha, \beta \diamond  \gamma ]_E-[\alpha, \beta]_E \diamond \gamma- \beta \diamond [\alpha, \gamma]_E \ , \]
it satisfies \[\mathcal L(\alpha \diamond \beta, \gamma, \tau)= \alpha \diamond \mathcal L(\beta , \gamma,\tau)+ \beta \diamond \mathcal L (\alpha,\gamma, \tau) \ , \]
that is the Leibnizator is a derivation in its first entry;

\item homomorphism  $\rho(\alpha \diamond \beta)= \rho(\alpha) \circ \rho(\beta)$ \ ;

\item Leibniz rule,  $[\alpha, f\, \beta]_E=(\rho(\alpha)\,f) \beta+ f\,[\alpha, \beta]_E\ ;$

\item Lie algebra homomorphism, $\rho([\alpha, \beta]_E)=[\rho(\alpha), \rho(\beta)] \ .$\\[0.5cm]
\end{enumerate}

That is we have the following map
\[   (E, \diamond, [\; , \;  ]_E, \mathcal U ) \xrightarrow{\;\;\;\;\;\;\rho \;\;\;\;\;\;} (TM, \circ,[\; , \;  ], e)\]
satisfying a), b), c) and d) above. Note that the triple $(TM, \circ, e)$ is an F-manifold.\\
\end{definition}

Now, it is known that there is a one-to-one correspondence between vector bundles (of finite rank) and locally free sheaves of finite rank. In other words, a vector bundle $E$ over a smooth manifold $M$ can be equivalently described as a sheaf $\mathcal E_M$ of locally free modules of finite rank over the ring $C^{\infty}(M)$ of smooth functions. Following this idea, we get: 

\begin{definition}
A smooth F-algebroid is the data  $(M, \mathcal E_M, \diamond, \mathcal U,  [\;, \;]_{\mathcal E}; \rho)$ consisting of a smooth manifold $M$ and a sheaf $\mathcal E_M$ of locally free modules of finite rank over the ring $C^{\infty}(M)$ which for every open $U \subset M$ the module $\mathcal E(U)$ has been endowed with the structure of an F-algebra, with operations: $\diamond$ denoting an associative and commutivative multiplication, with an unit $\mathcal U$; and a Lie bracket $[\;, \;]_{\mathcal E}$. Moreover, we have a homomorphism $\rho:  (\mathcal E_M, \diamond, \mathcal U,  [\;, \;]_{\mathcal E}) \to (\mathcal T_M, \circ, e,  [\;, \;])$ of F-algebra structures satisfying the Leibniz rule:
\[   [\alpha, f\, \beta]_{\mathcal E}=(\rho(\alpha)\,f) \beta+ f\,[\alpha, \beta]_{\mathcal E}\ ,\]

where $\alpha, \beta \in \mathcal E(U)$ and $f \in C^\infty (U)$.

\end{definition}

\begin{remark}
In a similar way we can define a complex F-algebroid, an analytic F-algebroid and an algebraic F-algebroid just changing the category of smooth manifolds by the corresponding complex, analytic or algebraic category. \\
\end{remark}

The concept of an F-algebroid can be easily extended to the realm of superspaces if we replace the smooth (or complex, or analytic or algebraic) manifold in its definition for a  smooth (or complex, or analytic or algebraic) supermanifold. Remember that \cite{V, CCF}: a supermanifold $\mathcal M$ of dimension $(m,n)$ is a locally ringed space $(M, \mathcal O_{\mathcal M})$ where $M$ is a manifold and $\mathcal O_{\mathcal M}$ is a sheaf of super-algebras that is locally isomorphic to  $C^\infty(\mathbb R^m) \otimes \Lambda^\cdot(\theta^1, \cdots, \theta^n)$ with $\Lambda^\cdot(\theta^1, \cdots, \theta^n)$ a Grassmann algebra in $n$ odd generators.
We find the definition of a super F-algebroid basically adding the ``super'' in several structures of the definition above. That is:

\begin{definition}
A super F-algebroid is the data  $(\mathcal M, \mathcal E_{\mathcal M}, \diamond, \mathcal U,  [\;, \;]_{\mathcal E}; \rho)$ consisting of a supermanifold $\mathcal M$ and a sheaf $\mathcal E_{\mathcal M}$ of locally free super modules over $\mathcal O_{\mathcal M}$ which for every open $U \subset \mathcal M$ the module $\mathcal E(U)$ has been endowed with the structure of a super F-algebra, with operations: $\diamond$ denoting an associative and super commutivative multiplication, with an unit $\mathcal U$; and a super Lie bracket $[\;, \;]_{\mathcal E}$. Moreover, we have a homomorphism $\rho:  (\mathcal E_{\mathcal M}, \diamond, \mathcal U,  [\;, \;]_{\mathcal E}) \to (\mathcal T_{\mathcal M}, \circ, e,  [\;, \;])$ of super F-algebra structures satisfying the Leibniz rule:
\[   [\alpha, f\, \beta]_{\mathcal E}=(\rho(\alpha)\,f) \beta+ f\,[\alpha, \beta]_{\mathcal E}\ ,\]

where $\alpha, \beta \in \mathcal E(U)$ and $f \in C^\infty (U)$.

\end{definition}

\subsection{The Pair $(\mathcal E (U) , C^\infty(U))$}
It is possible to encode algebraically the notion of super F-algebroid in terms of the concept of F-R pair as the following theorem shows.

\begin{theorem}
Let $(\mathcal M, \mathcal E_{\mathcal M}, \diamond, \mathcal U,  [\;, \;]_{\mathcal E}; \rho)$ be a super F-algebroid. Then, for every open $U\subset \mathcal M$, the pair $(\mathcal E (U) , C^\infty(U))$ is an F-alg--Rinehar pair.
\end{theorem}
\begin{proof}
We know that:
\begin{itemize}
\item The ring $C^\infty(U)$ of smooth functions on $U \subset \mathcal M$ is a commutative associative algebra over $\mathbb R$;

\item from the  definition of super F-algebroid, the module $\mathcal E(U)$ over the ring $C^\infty(U)$ has the structure of a super F-algebra. 
\end{itemize}
Note that the second item says that the super F-algebra $\mathcal E(U)$ is a $C^\infty(U)$-module, that is  there exist an action $f \act{c} \alpha$ of $f \in C^\infty(U)$ on $\alpha \in \mathcal E(U)$. Then the only missing part is to show that $C^\infty(U)$ is a $\mathcal E(U)$-module and check the compatibility conditions between modules. \\
Let   $\alpha \in \mathcal E(U)$ and $f \in C^\infty(U)$. We define the actions $\alpha \act{l} f$ and $\alpha \act{a} f$  by first applying the anchor map and then using a canonical action of super vector fields on functions which in the first case is an even super vector field and in the second an odd super vector field.  In other words, let $X$ be a super vector field such that $X=\rho(\alpha)$. Denoting by $D^{\text{even}}_X$ the even part of the derivation that represent the super vector field X over the ring of functions $C^\infty(M)$  and by  $D^{\text{odd}}_X$ the odd part of the derivation, we have:
\begin{align}
\alpha \act{l} f=&D^{\text{even}}_X f \ , \\  
\alpha \act{a} f=&D^{\text{odd}}_X f \ .
\end{align}
On $U \subset \mathcal M$ with super coordinates $(t^1, \cdots, t^m;\theta^1, \cdots, \theta^n)$ on $\mathbb R^{m|n}$, we have the following coordinate representation of $D_X$ , $D^{\text{even}}$ and $D^{\text{odd}}$
\begin{align}
D_X=& \sum_{\mu=1}^m X^\mu(t, \theta) \frac{\partial}{\partial t^\mu}+ \sum_{i=1}^n X^i(t, \theta) \frac{\partial}{\partial \theta^i} \ , \\
D^{\text{even}}_X=& \sum_{\mu=1}^m X^\mu(t, \theta) \frac{\partial}{\partial t^\mu} \  , \\
D^{\text{odd}}_X=& \sum_{i=1}^n X^i(t, \theta) \frac{\partial}{\partial \theta^i} \ .
\end{align}

Finally, the compatibility conditions of the modules are easily checked. 
\end{proof}


\section{Final Remarks}

In \cite{Hi, A} the notion of a Frobenius manifold was recovered using the concept of Higgs pair which is an extension of that of a Higgs bundle. On the other hand, we can define holomorphic F-algebroids working in the category of complex manifolds. In particular, we can consider working with Riemann surfaces. Following this ideas, we expect that  holomorphic F-algebroids structures over a Riemann surfaces should be closely related to Higgs bundles (this is work in progress). \\ 

Related to the study of F-algebroids over Riemann surfaces it would be needed to study more deeply the concept of super F-algebroid and its applications. This should be done in analogy with the developments of super Frobenius manifolds due to Manin and his collaborators. In this paper we have shown that the concept of super F-algebroid is crucial in order to get an algebraic description in terms of F-R pairs. This is an interesting example of the duality between algebraic and geometric objects that deserves further study. \\

Following Polischuck \cite{P} it should be possible to define an F-structure over a scheme. We think that an interesting problem to tackle is: \\

{\bf Problem 1}. Let $X$ be an F-scheme of finite type over $\mathbb{C}$, and $X_{red}$ be the corresponding reduced scheme. Then prove that $X_{red}$ and all its irreducible components are F-subschemes of $X$. \\

An affirmative answer to this problem will allow us to see the affine version of a super F-algebroid as a F-R pair. In addition, the concept of F-structure over a scheme could also be extended to superschemes following \cite{CCF, V}. We think that this approach will be useful in order to get a more algebraic geometric approach for the concept of Frobenius manifolds. This is work in progress.\\

It would be very interesting to describe the notion of ``F-groupoid'', that is, the global object  whose infinitesimal counterpart is an F-algebroid. The natural problem to solve is: \\

{\bf Problem 2}. Find conditions to determine when an F-algebroid is integrable. \\

In the case of a Lie algebroid Crainic and Fernandez showed \cite{CF} that the integrability problem is controlled by two computable obstructions, so the natural way to approch the above problem is to see whether Cranic and Fernandez conditions extend to the case of F-algebroids. It might be possible that the case of F-algebroids requires a novel approach. This is also work in progress. \\

Finally, we are currently constructing an F-alg--Rinehart operad whose algebras are F-R pairs. This is a color operad similar to the Lie-Rinehart operad described in \cite{L}. The F-alg--Rinehart operad has to be closely related to the $\bf{FMan}$ operad defined by Dotsenko \cite{D} and the minimal resolutions of Lie-Bialgebra operads and Gerstenhaber operads described by Merkulov in \cite{Me1, Me2}. These results will be presented in a forthcoming paper.

 

\begin{appendices}


\section{Lie--Rinehart Pair}\label{LRpair}

The notion of Lie algebroid can be encoded in the algebraic structure of a L-R pair. Some references on L-R pairs are \cite{Hu, Ma}.

\begin{definition}
A L-R pair denoted by the couple $(L, C)$ consist of 
\begin{itemize}
 
\item a Lie algebra $L$ with Lie bracket $[\; , \;]$ and 

\item a commutative associative algebra $C$ with product $\cdot$, 
 
\end{itemize}
such that 
\begin{enumerate}[i)]
\item $C$ is an $L$-module, that is, for every $X, Y \in L$ and $f \in C$ we have
\begin{align*}
L \times C &\to C  \\
(X, f)  & \mapsto X \act{l}f
\end{align*}
satisfying (compatibility between the Lie product and the Lie algebra action)
\[  [X,Y] \act{ l} f= X  \act{l}( Y \act{l} f )- Y  \act{l}( X \act{l} f) \ ,  \]

\item L is an $C$-module, that is, for every $f, g \in C$ and $X \in L$ we have
\begin{align*}
C \times L &\to L  \\
(f, X)  & \mapsto f \act{c} X
\end{align*}
satisfying (compatibility between the commutative associative product and the commutative associative algebra action)
\[  (f\cdot g) \act{c} X= f \act{c} (g \act{c} X)   \ ;\]

\end{enumerate}
and with the following compatibility between modules
\begin{enumerate}
\item compatibility between the Lie product and the commutative associative algebra action

\[ [X, f \act{c} Y ]= (X \act{l} f) \act{c} Y + f \act{c} [X, Y]  \ , \]

\item compatibility between the commutative associative product and the Lie algebra action

 \[ f \cdot (X \act{l} g)= (f \act{c}X) \act{l} g  \ .\]  \\

\end{enumerate}
\end{definition}

Given a Lie algebroid $(E\to M, [ \, , \, ]_E, \rho)$, where $E$ is a vector bundle over a manifold $M$, $[ \, , \, ]_E$ is a Lie bracket on sections $\Gamma(E)$ of the vector bundle and $\rho: \Gamma(E) \to \Gamma(TM)$ is the anchor map, we can naturally construct a L-R pair $(L, C)$ as follows: 
\begin{itemize}
\item let $L=\Gamma(E)$ be the Lie algebra of sections of the vector bundle $E$;

\item let $C=\mathcal C^\infty(M)$ be the commutative algebra of smooth functions on the manifold $M$;

\item the Lie action $\act{l}$ of $L=\Gamma(E)$ on $C=\mathcal C^\infty(M)$ is given by first applying the anchor map $\rho$ and then using the canonical action of vector fields on functions;

\item the commutative action $\act{c}$ of $C=\mathcal C^\infty(M)$ on $L=\Gamma(E)$ is the usual multiplication of sections of vector bundles over $M$ by functions on $M$.\\[0.6cm]
\end{itemize}

\end{appendices}

\textbf{\large Acknowledgements} \\

A. Torres-Gomez is partially supported by Universidad del Norte grant number 2018-17 ``Agenda I+D+I". J.A. Cruz Morales wants to thank the Institute Des Hautes \'Etudes Scientifiques (IHES) for its hospitality and support during his visit in December 2018 where part of this work was done. A. Torres-Gomez also would like to thank the Max Planck Institute for Mathematics (MPIM) for its hospitality and stimulating atmosphere during his visit in December 2018; part of this paper was completed during his stay at MPIM. \\[0.5cm]


------------------------------\\
John Alexander Cruz Morales,\\
Departamento de Matem\'aticas, Universidad Nacional de Colombia,\\
Bogot\'a, Colombia,\\ 
jacruzmo@unal.edu.co\\ [0.3cm]

Javier A. Gutierrez\\
Departamento de Matem\'aticas, Universidad Sergio Arboleda,\\
Bogot\'a, Colombia,\\ 
javier.gutierrez@usa.edu.co\\[0.3cm]

Alexander Torres-Gomez\\
Departamento de Matem\'aticas y Estad\'istica, Universidad del Norte,\\
Barranquilla, Colombia,\\ 
alexander.torres.gomez@gmail.com

\end{document}